\documentclass{article}

\usepackage{amsfonts,amssymb}

\usepackage{pinlabel}
\input{epsf.tex}
\usepackage{latexsym,amsfonts,amssymb,verbatim,mathrsfs,amsthm}
\usepackage{amsmath,amsthm,amssymb,latexsym,graphics,textcomp,graphicx}
\usepackage{eucal,eufrak}
\usepackage{colonequals}
\usepackage{graphicx}
\usepackage{url}
\input{xy}
\xyoption{all}

\theoremstyle{plain}
\newtheorem{theorem}{Theorem}[section]
\newtheorem{lemma}[theorem]{Lemma}
\newtheorem{corollary}[theorem]{Corollary}
\newtheorem*{theoremmain}{Main Theorem}

\newtheorem{conjecture}[theorem]{Conjecture}

\newtheorem{proposition}[theorem]{Proposition}

\newtheorem*{claim}{Claim}

\theoremstyle{definition}

\newtheorem{question}{Question}

\newcommand{\supp}{\mathrm{supp}}
\newcommand{\cone}{\mathrm {cone}}

\newcommand{\curr}{\mathrm{Curr}}
\newcommand{\Flat}{\mathrm{Flat}}

\newcommand{\R}{\mathbb{R}}
\newcommand{\Z}{\mathbb{Z}}

\newcommand{\Ch}{\mathrm{Chain}}

\newcommand{\G}{\mathcal{G}}
\newcommand{\CAT}{\mathrm{CAT}}

\begin{document}

\title{Marked length spectral rigidity for flat metrics.}
\author{Anja Bankovic and Christopher J Leininger\thanks{The second author was partially supported by the NSF grant DMS-1207183}}

\maketitle

\textbf{Abstract}:  In this paper we prove that the space of flat metrics (nonpositively curved Euclidean cone metrics) on a closed, oriented surface is marked length spectrally rigid.  In other words, two flat metrics assigning the same lengths to all closed curves differ by an isometry isotopic to the identity.

\section{Introduction}

Let $S$ be a closed, orientable surface and  $\mathfrak{M}(S)$ a set of metrics on $S$, defined up to isometry isotopic to the identity.  For $m\in\mathfrak{M}(S)$, we denote the {\em marked length spectrum of $m$} by
\[ \Lambda(m)=\{\ell_m(\gamma)\}_{\gamma\in\mathcal{C}(S)},\] where $\mathcal{C}(S)$ denotes the set of homotopy classes of non-null homotopic closed curves on $S$ and $\ell_m(\gamma)$ the length of a minimal $m$--geodesic representative of $\gamma$. We say that $\mathfrak{M}(S)$ is {\em spectrally rigid} if $m \mapsto \Lambda(m)$ is injective on $\mathfrak M(S)$.  Let $\Flat(S)$ denote the space of nonpositively curved Euclidean cone metrics.  In this paper we prove the following

\begin{theoremmain} \label{T:main1}
If $\varphi_1,\varphi_2 \in \Flat(S)$ and $\Lambda(\varphi_1) = \Lambda(\varphi_2)$, then $\varphi_1 = \varphi_2$.
\end{theoremmain}

The first results on spectral rigidity for surfaces are due to Fricke and Klein \cite{FrickeKlein}, who showed that the Teichm{\"u}ller space of Riemannian metrics with constant curvature $-1$ is spectrally rigid.  Otal \cite{otal:length} generalized this and showed that set of all negatively curved Riemannian metrics is spectrally rigid (see also Croke \cite{croke}). This was further generalized in two directions, first by Croke-Fathi-Feldman \cite{Croke:TM} who proved that the space of non--positively curved Riemannian metrics is spectrally rigid, and second by Hersonsky-Paulin \cite{hersonsky:OT} who showed that negatively curved {\em cone metrics} are spectrally rigid.  Frazier shows in \cite{Frazier:LS} that $\Lambda(\varphi)$ distinguishes metrics $\varphi \in \Flat(S)$ from negatively curved Riemannian and negatively curved cone metrics (and in fact, from nonpositively curved Riemannian metrics).   Duchin, Leininger, Rafi \cite{rafi:LD} showed that the subset of metrics in $\Flat(S)$ coming from quadratic differentials is spectrally rigid. 

To prove the Main Theorem, we follow \cite{otal:length,Croke:TM,hersonsky:OT}, associating the {\em Liouville geodesic current} $L_\varphi$ to each metric $\varphi \in \Flat(S)$, which has the property that for every closed curve, its length is calculated via Bonahon's intersection number with $L_\varphi$ \cite{bonahon:TC}.  Appealing to a result of Otal \cite{otal:length}, recorded as Theorem~\ref{theorem:Otal} here, reduces the proof to proving that two metrics with the same Liouville current are equivalent.   Our analysis diverges at this point as we focus almost exclusively on the {\em support} of the measure $L_\varphi$.  The very special behavior of the support was evident in \cite{rafi:LD}, and played a key role in \cite{Frazier:LS}.   We are able to determine a great deal about the metric from the support alone, and in fact, we conjecture that up to an obvious ambiguity, the support of $L_\varphi$ determines $\varphi$; see Section~\ref{sec:questions}.

The outline of the paper is as follows. In Section~\ref{sec:EuclideanCone} we describe Euclidean cone metrics, the induced metrics in the universal cover, and prove various facts about their geodesics.  In Section~\ref{sec:currents} we define geodesic currents and the Liouville current for a flat metric.  Section~\ref{sec:chains} provides key relationships between the Liouville current, cone points in a flat metric, and distances between cone points. The proof of the Main Theorem is given in Section~\ref{sec:end}.  We end with a conjecture and question in Section~\ref{sec:questions}.

\bigskip

\noindent {\bf Acknowledgements.}  We would like to thank Stephanie Alexander for helpful suggestions and the reference to Reshetnyak's Theorem used in Proposition~\ref{prop:CAT(0)}.

\section{Euclidean cone metrics and their geodesics}
\label{sec:EuclideanCone}

\subsection{$\CAT(0)$ geometry}

A geodesic metric $\varphi$  on $S$ is called a \emph{Euclidean cone metric} if there is a finite set of points on $S$, denoted $\cone(\varphi)$, so that the following hold: 

\begin{enumerate}

\item[(i)] $\varphi$ is locally isometric to $\mathbb{R}^2$ with the Euclidean metric on $S\setminus \cone(\varphi)$, and

\item[(ii)] every point in $\cone(\varphi)$ has an $\epsilon$--neighborhood isometric to the metric space obtained by gluing together some (finite) number of sectors of $\epsilon$--balls about 0 in $\mathbb{R}^2$ by isometries.

\end{enumerate}

The points in $\cone(\varphi)$ will be called the {\em cone points} of $\varphi$. Each  $\zeta\in\cone(\varphi)$ has a well defined cone angle $\mathfrak{ang}(\zeta)$, which is equal to the sum of the angles of the sectors from (ii) above. We can extend the definition of angle on non--cone points by defining  $\mathfrak{ang}(\zeta)=2\pi$ for all $\zeta \in S\setminus \cone(\varphi)$.  

Every Euclidean cone surface $S$ has triangulation (more precisely $\Delta$--complex structure) for which the vertex set is precisely the set of cone points; see, e.g.~\cite{masur:HDN}.  By Gromov's link condition, it follows that a Euclidian cone metric $\varphi$ is {\em nonpositively curved} if and only if $\mathfrak{ang}(\zeta)> 2\pi$ for every point $\zeta\in\cone(\varphi)$; see \cite[Theorem II.5.2]{bridson:NPC}.  A nonpositively curved Euclidean cone metric will be called a {\em flat metric}.  The space of all flat metrics on $S$, up to isometry isotopic to the identity is denoted $\Flat(S)$.  We will not distinguish between a metric and its equivalence class.

We will eventually use the following to construct our isometry in the proof of the Main Theorem. 

\begin{proposition} \label{prop:CAT(0)}
Suppose $\Delta$ is a geodesic triangle in a complete, locally compact, CAT(0) Euclidean cone surface $X$ and $\Delta' \subset \mathbb R^2$ is its comparison triangle.  Then $Area(\Delta) \leq Area(\Delta')$ with equality if and only if $\Delta$ is isometric to $\Delta'$.
\end{proposition}

\begin{proof} Let $T$ be the $2$--simplex bounded by $\Delta'$ in $\R^2$.   According to Reshetnyak's Majorization Theorem \cite{reshetnyak},\cite[Chapter 9.8]{Alexander:AG}, the comparison path isometry from $\Delta'$ to $\Delta$ extends to a $1$--Lipschitz map $f \colon T \to X$ into the convex hull of $\Delta$.  Therefore, we must show that $Area(f(T)) \leq Area(T)$ with equality if and only if $f$ is an isometry.  However, the $1$--Lipschitz map is area non-increasing, and it is area preserving if and only if it is an isometry.  This is obvious for smooth surfaces $X$, and follows in our slightly more general case since $f(T)$ contains at most finitely many cone points.
\end{proof}





\subsection{Spaces of geodesics}

Let $p \colon \widetilde S \to S$ denote the universal cover.  For any geodesic metric $\sigma$ on $S$, we use the same name $\sigma$ to denote the induced geodesic metric on $\widetilde S$.  Let $S^1_\infty$ denote the circle at infinity of $\widetilde S$ --- equivalently, the Gromov boundary of $\widetilde S$ --- with respect to $\sigma$.  This compactifies $\widetilde S$ to a closed disk, and the action of $\pi_1(S)$ on $\widetilde S$ extends to an action by homeomorphisms on this disk.  Any other geodesic metric $\sigma'$ on $S$ induces its own circle at infinity, but the identity on $\widetilde S$ extends to a homeomorphism between the corresponding closed disks, and so we view $S^1_\infty$ as the boundary of $\widetilde S$, independent of $\sigma$; see \cite[Chapter III.H.3]{bridson:NPC}.

Let $\mathcal G(\sigma)$ denote the space of bi-infinite $\sigma$--geodesics in $\widetilde S$.  This is given as the quotient of the space of unit speed parameterized geodesics with the compact-open topology, where we forget the parameterization.   We record the endpoints-at-infinity of any $\delta \in \G(\sigma)$ and denote it $\partial_\sigma(\delta) = \{x,y\}$.  We view this as an unordered pair of points; that is, an element of
\[ \mathcal{G}(\widetilde S) = (S^1_\infty \times S^1_\infty \setminus D)/_{(x,y) \sim (y,x)} \]
where $D$ is the diagonal, $D=\{(x,x)\, | \, x\in S^1_\infty\}$.   The function $\partial_\sigma$ is a continuous, $\pi_1(S)$--equivariant surjective map
\[ \partial_\sigma \colon \mathcal G(\sigma) \to \mathcal G(\widetilde S).\]
When $\sigma$ is a negatively curved metric, $\partial_\sigma$ is a homeomorphism, but for a general metric $\sigma$, it need not be.

\begin{proposition} \label{P:closed map}
For any geodesic metric $\sigma$, the map $\partial_\sigma$ is a closed map.
\end{proposition}
\begin{proof}
Let $E \subset \G(\sigma)$ be a closed set, and we must show that $\partial_\sigma(E)$ is a closed set in $\G(\widetilde S)$.  Since $\mathcal G(\widetilde S)$ is metrizable, it suffices to show that if $\{\delta_n\}_{n=1}^\infty \subset E$ is a sequence such that $\partial_\sigma(\delta_n)$ converges to some $\{x,y\}$, then there is some $\delta \in E$ such that $\partial_\sigma(\delta) = \{x,y\}$.  For this, fix a point $\zeta \in \widetilde S$ and observe that since $\partial_\sigma(\delta_n)$ converges to $\{x,y\}$ (and $x \neq y$), the distance from $\zeta$ to the geodesics $\delta_n$ is bounded by some constant $R > 0$, independent of $n$.  Since the metric $\sigma$ on $\widetilde S$ is proper, the closed ball $B_R$ of radius $R> 0$ about $\zeta$ is compact.  Since the $\delta_n$ all intersect $B_R$ for all $n$, the Arzela-Ascoli Theorem implies that some subsequence $\{\delta_{n_k}\}$ converges to some geodesic $\delta$. Since $E$ is closed, $\delta \in E$, and continuity of $\partial_\sigma$ implies $\partial_\sigma(\delta) = \{x,y\}$.
\end{proof}



\subsection{Linking and betweenness} \label{sec:link-between}

Fix a hyperbolic metric $\rho$ inducing a homeomorphism $\partial_\varphi \colon \mathcal G(\rho) \to \mathcal G(\widetilde S)$.
Let $\delta_1,\delta_2 \in \mathcal G(\rho)$ be two distinct geodesics, and write $\partial_\rho (\delta_i) = \{ x_i,y_i\}$, for $i=1,2$.  If $\delta_1,\delta_2$ transversely intersect, then we say that $\{x_1,y_1\}$ and $\{x_2,y_2\}$ {\em link}.  This is equivalent to saying that the $0$--spheres $\{x_1,y_1\}$ and $\{x_2,y_2\}$ in the $1$--sphere $S^1_\infty$ are linked, meaning that $x_1,x_2,y_1,y_2$ are all distinct, and the two components of $S^1_\infty \setminus \{x_1,y_1\}$ each contain one of the points $x_2,y_2$.  The point is that intersection of $\delta_1,\delta_2$ can be determined from the image in $\partial_\rho$. We will also say that $\delta_1,\delta_2$ link.

Suppose $\delta_1,\delta_2$ are disjoint (so their endpoints do not link).  Then write $[\delta_1,\delta_2] \subset \G (\rho)$ for the set of geodesics {\em between} $\delta_1,\delta_2$.  These are precisely the geodesics that link neither of $\delta_1$ nor $\delta_2$ but do link every geodesic which is linked with both $\delta_1$ and $\delta_2$; See Figure \ref{F:linkingbetween}.  Equivalently, if we let $\partial_\sigma(\delta_i) = \{x_i,y_i\}$, $i = 1,2$, with the points appearing in $S^1_\infty$ in the counterclockwise, cyclic order $x_1,y_1,y_2,x_2$, and if we let $[a,b]$ denote the counterclockwise interval between $a$ and $b$ in $S^1_\infty$ (where $[a,b]=\{a\}$ if $a =b$), then
\[ [\delta_1,\delta_2] = \{ \delta \in \G(\rho) \mid \partial_\rho (\delta) = \{x,y\} \mbox{ with } x \in [x_2,x_1], y \in [y_1,y_2] \}. \]
The image of $[\delta_1,\delta_2]$ in $\G(\widetilde S)$ is similarly denoted
\[ \big[\{x_1,y_1\},\{x_2,y_2\} \big] = \partial_\sigma([\delta_1,\delta_2]) = \{\{x,y\} \mid x \in [x_2,x_1], y \in [y_1,y_2] \}. \]

For $\varphi \in \Flat(S)$, the endpoints of $\delta_1,\delta_2 \in \G(\varphi)$ link if and only if either $\delta_1,\delta_2$ transversely intersect once, or they share a compact segment, with $\delta_1$ crossing from one side of $\delta_2$ to the other.   In this case, we'll also say that $\delta_1,\delta_2$ link.  Betweenness for $\delta_1,\delta_2 \in \G(\varphi)$ is also defined as betweenness for the endpoints.

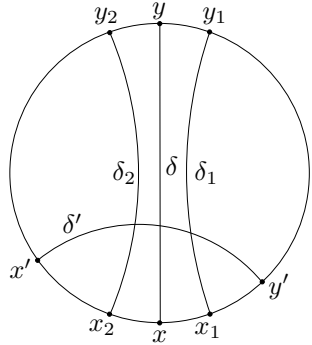
\begin{figure}[htb]
\centering
\includegraphics[width=4cm]{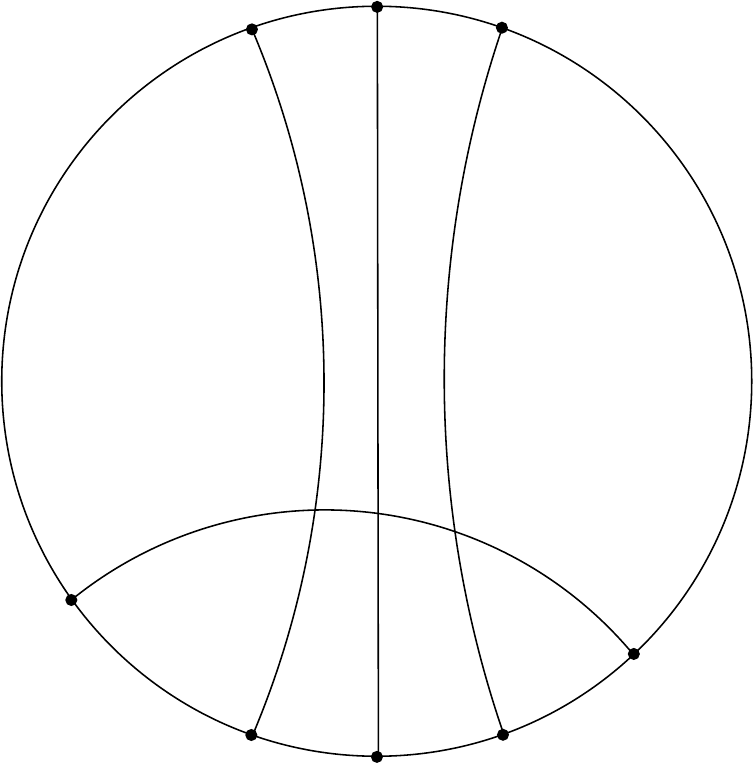}
\begin{picture}(0,0)
\put(-63,118){$y$}
\put(-63,-6){$x$}
\put(-44,116){$y_1$}
\put(-87,-2){$x_2$}
\put(-85,116){$y_2$}
\put(-47,-3){$x_1$}
\put(-78,55){$\delta_2$}
\put(-47,55){$\delta_1$}
\put(-58,57){$\delta$}
\put(-97,35){$\delta'$}
\put(-117,18){$x'$}
\put(-19,12){$y'$}
\end{picture}
\caption{$\delta_1,\delta_2$ do not link, and $\delta \in [\delta_1,\delta_2]$ since any geodesic $\delta'$ linking $\delta_1,\delta_2$ will also link $\delta$.} \label{F:linkingbetween}
\end{figure}

\subsection{Flat metric geodesics}

Let $\varphi \in \Flat(S)$ be any flat metric.   A $\varphi$--geodesic $\delta \in \G(\varphi)$ is called {\em nonsingular} if it contains no cone points.  We let $\G^\circ(\varphi) \subset \G(\varphi)$ denote the set of nonsingular geodesics, and $\G^*(\varphi) =  \overline{\G^\circ(\varphi)} \subset \G(\varphi)$.  Since there are countably many cone points of $\varphi$ in $\widetilde S$, and the set of $\varphi$--geodesics through any cone point is closed, it follows that $\G^\circ(\varphi)$ is the complement of a countable union of closed sets, and hence a Borel set.  Each $\delta \in \G(\varphi)$ is two-sided and we choose a transverse orientation so that we can refer to the positive and negative sides of $\varphi$.   At every cone point $x \in \cone(\varphi) \subset \widetilde S$, $\delta$ makes two angles, one on the positive side and one on the negative side.  Because $\delta$ is a geodesic, at every $x \in \cone(\varphi)$, both angles at $x$ must be at least $\pi$.

\begin{proposition}
If $\delta \in \G(\varphi)$ is a geodesic containing at most one cone point and making an angle $\pi$ on one side at that point, then $\delta \in \G^*(\varphi)$.
\end{proposition}
\begin{proof}
Suppose $\delta$ makes an angle $\pi$ on the positive side at $\zeta \in \cone(\varphi)$.  Fix a short reference geodesic arc $\alpha$ emanating from $\zeta$ on the positive side of $\delta$, orthogonal to $\delta$, meeting no other cone points.  At every point $\alpha(t)$, we consider the geodesic through $\alpha(t)$ orthogonal to $\alpha$, and hence parallel to $\delta$.  Since these are all parallel along $\alpha$ and $\varphi$ is $\CAT(0)$, these are pairwise disjoint.  At most countably many of these can meet a cone point (since there are only countably many cone points), and hence $\delta$ is a limit of nonsingular geodesics approaching it from the positive side.  Therefore $\delta \in \G^*(\varphi)$.
\end{proof}

\begin{proposition} \label{P:at most one switch}
If $\delta \in \G^*(\varphi)$, then at every cone point $\zeta \in \delta$, $\delta$ makes an angle exactly $\pi$ on one side.  Furthermore, ordering the cone points linearly along $\delta$, the side on which the angle is $\pi$ can only switch at most once.
\end{proposition}
\begin{proof}
For the first claim, let $\delta_n \in \G^\circ(\varphi)$ be a sequence converging to $\delta$, and $\zeta \in \cone(\varphi)$ contained in $\delta$.  Since $\delta_n$ contains no cone points, $\zeta \not\in \delta_n$, and therefore up to subsequence $\delta_n$ approaches $\delta$ from either its positive or negative side {\em near $\zeta$}.  In this case, the cone angle of $\delta$ at $\zeta$ on the side of approach must be $\pi$.  See Figure~\ref{F:limit sides}.

If $\zeta_1,\zeta_2,\zeta_3$ are consecutive cone points along $\delta$, and the angle $\pi$ is on the positive side at $\zeta_1$ and $\zeta_3$ and on the negative side at $\zeta_2$, say, then it follows that for $n$ sufficiently large, an approximating geodesic $\delta_n \in \G^\circ(\varphi)$ for $\delta$ must be on the positive side near $\zeta_1$ and $\zeta_3$ and the negative side near $\zeta_2$; see Figure~\ref{F:limit sides}.  But then $\delta_n$ crosses $\delta$ twice, which is impossible since $\varphi$ is $\CAT(0)$.
\end{proof}

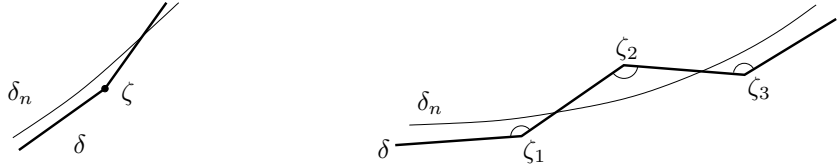
\begin{figure}[htb]
\centering
\includegraphics[width=11cm]{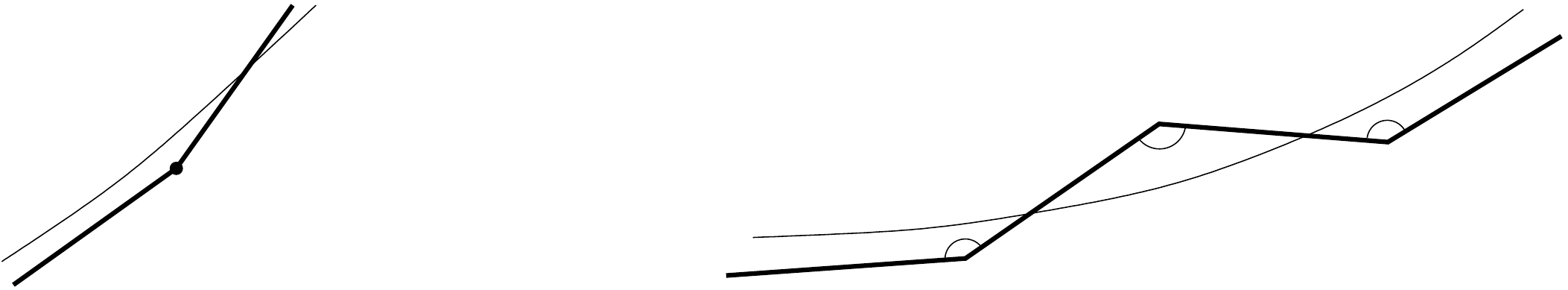}
\begin{picture}(0,0)
\put(-275,20){$\zeta$}
\put(-293,0){$\delta$}
\put(-318,20){$\delta_n$}
\put(-123,-2){$\zeta_1$}
\put(-88,37){$\zeta_2$}
\put(-38,22){$\zeta_3$}
\put(-163,16){$\delta_n$}
\put(-178,-2){$\delta$}
\end{picture}
\caption{Approximating $\delta$ by $\delta_n \in \G^\circ(\varphi)$.  {\bf Left}: Side of $\delta$ near $\zeta$ containing $\delta_n$ must make angle $\pi$.  {\bf Right}:  The points $\zeta_1,\zeta_2,\zeta_3$ have cone angle $\pi$ on the side indicated, switching sides from $\zeta_1$ to $\zeta_2$ and again from $\zeta_2$ to $\zeta_3$.} \label{F:limit sides}
\end{figure}

If we define $\G^2(\varphi) \subset \G^*(\varphi)$ to be the set of geodesics in $\G^*(\varphi)$ containing at least two cone points, we have the following corollary of the previous proposition.
\begin{corollary} \label{C:countable}
The set $\G^2(\varphi)$ is countable.
\end{corollary}
\begin{proof}
Any $\varphi$--geodesic containing more than one cone point must a contain a geodesic segment connecting a pair of cone points.  Furthermore, by Proposition~\ref{P:at most one switch}, every geodesic in $\G^*(\varphi)$ switches the sides with angle $\pi$ at most once.  Therefore, there are at most countably many geodesics in $\G^*(\varphi)$ containing a given geodesic segment between cone points.  On the other hand, there are countably many geodesic segments between cone points.  Thus $\G^2(\varphi)$ is a countable union of countable sets, hence countable.
\end{proof}

We are also interested in the $\partial_\varphi$--images of the subsets defined above, and we denote these
\[ \G^2_\varphi(\widetilde S) = \partial_\varphi(\G^2(\varphi)) \quad \quad \G^\circ_\varphi(\widetilde S) = \partial_\varphi(\G^\circ(\varphi)) \quad \quad  \G_\varphi^*(\widetilde S) = \partial_\varphi(\G^*(\varphi)). \]

\subsection{Asymptotic geodesics}

Geodesics in $\G^*(\varphi)$ can only be asymptotic in one or both directions in particularly special ways.  We will exploit this information, and so we describe this precisely.

A {\em $\varphi$--flat strip} in $\widetilde S$ is an isometric embedding $F \colon \R \times [a,b] \to \widetilde S$ for some $[a,b]$, and a {\em $\varphi$--flat half-strip} in $\widetilde S$ is an isometric embedding of $F \colon [0,\infty) \times [a,b] \to \widetilde S$ for some $[a,b]$.  In both cases we required  $a \neq b$.

The failure of injectivity of $\partial_\varphi$ is entirely accounted for by $\varphi$--flat strips.  More precisely, the fiber of $\partial_\varphi$ is either a single geodesic, or else a $\varphi$--flat strip; see \cite[Theorem II.2.13]{bridson:NPC}.  In particular, for any point $\{x,y\} \in \G(\widetilde S)$, $\partial_\varphi^{-1}(\{x,y\})$ is either a point or an arc.

We say that two geodesics $\delta_1,\delta_2 \in \G(\varphi)$ are {\em $\varphi$--cone point asymptotic} (in one direction) if there are rays $\delta_i^+ \subset \delta_i$, $i=1,2$ so that $
\delta_1^+ = \delta_2^+$.  If this happens, there is a maximal such ray in each which necessarily emanates from a cone point.

\begin{lemma} \label{L:ray-asympt-class}
Suppose $\delta_1,\delta_2 \in \G^*(\varphi) \setminus \G^2(\varphi)$ are asymptotic one one direction.  Then either $\delta_1,\delta_2$ are $\varphi$--cone point asymptotic, or else they are separated by a $\varphi$--flat half-strip.
\end{lemma}
\begin{proof}  If $\delta_1$ and $\delta_2$ do not share a common ray, consider the region bounded by their asymptotic rays $\delta_1^+$ and $\delta_2^+$.  Since these geodesics are in $\G^*(\varphi) \setminus \G^2(\varphi)$ we can choose these rays to contain no cone points.    Consider the convex hull of these two rays.   Note that there can be no cone points in this region as this would force the rays to diverge.  Therefore, this region can be embedded into $\R^2$ as the convex hull of two asymptotic Euclidean rays.  But such a convex hull must contain an isometrically embedded $[0,\infty] \times [0,\epsilon]$ for some $\epsilon > 0$.
\end{proof}

\begin{lemma} \label{L:cone point char}
Suppose $x,x_1,x_2 \in S^1_\infty$ three distinct points with $\{x,x_1\},\{x,x_2\} \in \G^*_\varphi(\widetilde S) \setminus \G^2_\varphi(\widetilde S)$.  Then the following conditions are equivalent.
\begin{enumerate}
\item There exists $\delta_1,\delta_2 \in \G^*(\varphi)$ with $\partial_\varphi(\delta_i) = \{x,x_i\}$ for $i =1,2$, such that $\delta_1,\delta_2$ are $\varphi$--cone point asymptotic.
\item $[\{x,x_1\}, \{x,x_2\} ] \cap \partial_\varphi(\mathcal G^*(\varphi)) = \{ \{x,x_1\},\{x,x_2\} \}$.
\end{enumerate}
When this happens, $\delta_1,\delta_2$ are unique and hence so is the cone point they contain.
\end{lemma}
Recall that $[\{x,x_1\}, \{x,x_2\} ]$ denotes the set of endpoints {\em between} $\{x,x_1\}$ and $\{x,x_2\}$; see Section~\ref{sec:link-between}.
\begin{proof}
The first condition implies the second: any geodesic $\delta$ strictly between $\delta_1,\delta_2$ (i.e.~between, and not equal to either $\delta_i$) would have to contain the common ray of these two geodesics.  Note that $\delta$ cannot agree with, say, $\delta_1$ beyond this ray, since there are no more cone points along it where $\delta$ could depart from $\delta_1$ (similarly for $\delta_2$).  But then $\delta$ cannot make an angle $\pi$ at the cone point.

Now suppose the $\{x,x_1\},\{x,x_2\}$ satisfy the second condition and let $\delta_1,\delta_2 \in \G^*(\varphi)$ be any two geodesics with $\partial_\varphi(\delta_i) = \{x,x_i\}$ for $i =1,2$.   Note that $\delta_1,\delta_2 \in \G^*(\varphi) \setminus \G^2(\varphi)$ by assumption.  These geodesics are asymptotic, and so by Lemma \ref{L:ray-asympt-class} they are either cone-point asymptotic or else there is a $\varphi$--flat half-strip between them.  If they are cone-point asymptotic we are done, so suppose there is a $\varphi$--flat strip between them.

Suppose first that there is an entire $\varphi$--flat strip between $\delta_1,\delta_2$.  Then if one of $\delta_i$ is a geodesic in this strip (possibly the boundary geodesic), then we can replace it by the other boundary geodesic so that this strip is no longer between $\delta_1,\delta_2$.  If neither $\delta_i$ is a geodesic in the strip, then any geodesic $\delta$ in the strip lies between $\delta_1,\delta_2$, and has endpoints $x,y \in S^1_\infty$ with $y \neq x_1,x_2$.  This contradicts condition 2, and therefore we may assume that there is no such flat strip.

There is still a $\varphi$--flat half-strip between $\delta_1,\delta_2$ by assumption.  In this case, it is easy to find a nonsingular geodesic with a ray in this half-strip between $\delta_1,\delta_2$ and having a lower bound on the distance to each.  The endpoints of this are $x,y \in S^1_\infty$ with $y \neq x_1,x_2$, another contradiction.

Since any two geodesics with the same endpoints must bound a flat strip, the proof shows that $\delta_1,\delta_2$ are unique.
\end{proof}

\section{Geodesic currents}
\label{sec:currents}

\subsection{Geodesic currents and intersection numbers}

The action of $\pi_1(S)$ on $\widetilde S$ (or on $S^1_\infty$) determines an action on $\mathcal{G}(\widetilde S)$.  A {\em geodesic current} is defined to be a $\pi_1(S)$--invariant Radon measure on $\mathcal{G}(\widetilde S)$.  The space of all geodesic currents on $S$ with the weak* topology is denoted $\curr(S)$.  Given a non-null homotopic closed curve $\gamma$ on $S$, the endpoints of components of the preimage in $\widetilde S$ is a discrete, $\pi_1(S)$--invariant set of points in $\G(\widetilde S)$.  The counting measure on this set defines a geodesic current on $S$ which we also denote $\gamma$.  The set of real multiples of such currents are dense in $\curr(S)$.  For a discussion of these facts, as well as the following theorem, see Bonahon~\cite{bonahon:ends,bonahon:TC}.

\begin{theorem}
There exists a continuous, symmetric, bilinear form
\[ i \colon \curr(S) \times \curr(S) \to \R\]
such that for any two currents $\gamma_1,\gamma_2$ associated to closed curves of the same name, $i(\gamma_1,\gamma_2)$ is the geometric intersection number of the homotopy classes of these curves.
\end{theorem}

For geodesic currents $\mu,\gamma$, where $\gamma$ is the current associated to a closed curve of the same name, $i(\gamma,\mu)$ can be calculated as follows; see \cite{bonahon:TC}.  Choose a component $\widetilde \gamma \subset \widetilde S$ of $p^{-1}(\gamma)$.  Let $I(\widetilde \gamma)$ be a $\mu$--measurable fundamental domain for the action of the stabilizer of $\widetilde \gamma$ in $\pi_1(S)$ on the subset of $\G(\widetilde S)$ consisting of pairs of points linking the endpoints of $\widetilde \gamma$.  Then $i(\gamma,\mu)$ is computed as the $\mu$--measure of $I(\widetilde \gamma)$:
\begin{equation} \label{E:intersection_length_calc}
i(\mu,\gamma)= \mu(I(\widetilde \gamma))
\end{equation}

The following result of Otal~\cite{otal:length} is an important ingredient in marked length spectral rigidity.

\begin{theorem} \label{theorem:Otal} Two currents $\mu_1,\mu_2\in\curr(S)$ are equal if and only if $i(\gamma,\mu_1)=i(\gamma,\mu_2)$ for every curve $\gamma\in \mathcal{C}(S)$. 
\end{theorem}

\subsection{Flat Liouville current}

Fix $\varphi \in \Flat(S)$.  Here we define the {\em pre-Liouville current} $\hat L_\varphi$ for $\varphi$ as a $\pi_1(S)$--invariant measure on $\G(\varphi)$ as follows as follows; see \cite{paternain,ALcurrents}.  First, we let $T^1(S^*)$ denote the unit tangent bundle over $S^* = S \setminus \cone(\varphi)$.  The (local) geodesic flow on $T^1(S^*)$, has a canonical invariant volume form given locally as one-half of the product of the area form on $S^*$ and the angle form on the fiber circles.  Contracting with the vector field generating the flow gives a flow-invariant $2$--form.  The absolute value is an invariant measure on the local leaf spaces of the foliation by flow lines.  Now lift to the universal cover, and restrict to the subspace where the flow is defined for all time.  The flow lines are precisely the geodesics in $\G^\circ(\varphi)$ (though they are oriented).  Thus the measure on the local leaf space determines a measure on the Borel set $\G^\circ(\varphi)$ which is invariant by $\pi_1(S)$.  This is extended by zero to the rest of $\G(\varphi)$.  Therefore the support of $\hat L_\varphi$ is contained in the closure $\G^*(\varphi)$  of $\G^\circ(\varphi)$.

There is a simple local formula for $\hat L_\varphi$ we now describe.  First, for any $\varphi$--geodesic segment $\alpha \subset \widetilde S$, parameterized by $t \mapsto \alpha(t)$ for $t \in [a,b]$, containing no cone points in its interior, consider the subset $E^\circ(\alpha) \subset \G^\circ(\varphi)$ of geodesics transversely crossing $\alpha$:
\[ E^\circ(\alpha) = \{ \gamma \in \mathcal G^\circ(\varphi) \mid \gamma \pitchfork \alpha \neq \emptyset \}. \]
Any geodesic $\gamma \in E^\circ(\alpha)$ is uniquely determined by $t \in [a,b]$ with $\alpha(t) = \gamma \cap \alpha$ and the angle $\theta \in (0,\pi)$ counterclockwise from $\alpha$ (with positive orientation) and $\gamma$.  Let $D^\circ(\alpha) \subset [a,b] \times (0,\pi)$ denote the set of pairs $(t,\theta)$ such that there exists a (necessarily unique) geodesic $\gamma(t,\theta) \in E^\circ(\alpha)$ meeting $\alpha$ at $\alpha(t)$ and making an angle $\theta$.  The assignment $(t,\theta) \mapsto \gamma(t,\theta)$ defines a bijection $D^\circ(\alpha) \to E^\circ(\alpha)$, which is easily seen to be a homeomorphism.
The measure $\hat L_\varphi$ restricted to $E^\circ(\alpha)$ is given by the push forward of the measure on $D^\circ(\alpha)$ given by
\begin{equation} \label{E:local-liouville}
\hat L_\varphi  = \frac{1}{2} \sin(\theta)  \, d \theta \, dt .
\end{equation}

The measure on the right is absolutely continuous with respect to Lebesgue measure on $[a,b] \times (0,\pi)$, and we note that $D^\circ(\alpha)$ is a set of full Lebesgue measure.  Indeed, for every $t$, and each of the countably many cone point in $\widetilde S$, there is exactly one geodesic segment from the cone point to $\alpha(t)$.  Consequently, $D^\circ(\varphi)$ is the complement of a countable set of closed sets, each intersecting the sets $\{t \} \times [0,\pi]$ in a single point, for each $t$.  Consequently, every point of $E^\circ(\alpha)$ is in the support of $\hat L_\varphi$.  It follows that every element of $\G^\circ(\varphi)$ is in the support, and hence $\supp(\hat L_\varphi) = \G^*(\varphi)$.  From (\ref{E:local-liouville}) we easily deduce the following.

\begin{proposition} \label{P:pre-measure is length}
For any geodesic segment $\alpha$, we have $\hat L_\varphi(E^\circ(\alpha)) = \ell_\varphi(\alpha)$, the $\varphi$--length of $\alpha$. \qed
\end{proposition}

Now we use the map $\partial_\varphi \colon \mathcal G(\varphi) \to \mathcal G(\widetilde S)$ to push this forward to currents, and declare this to be the {\em Liousville current} for $\varphi$:
\[ L_\varphi = \partial_{\varphi *} \hat L_\varphi.\]
Since the geodesic representative of a closed curve $\alpha$ is a union of finitely many Euclidean segments between cone points, the following is straightforward from Proposition~\ref{P:pre-measure is length} and Equation~(\ref{E:intersection_length_calc})
\begin{proposition} \label{P:measure is length}
For every closed curve $\alpha$ we have $i(L_\varphi,\alpha) = \ell_\varphi(\alpha)$. \qed
\end{proposition}
We also have the following corollary of Proposition~\ref{P:closed map}
\begin{corollary} \label{C:support is closure}
For any $\varphi \in \Flat(S)$, $\supp(L_\varphi) = \G^*_\varphi(\widetilde S)$.
\end{corollary}
\begin{proof}
A continuous, closed map always sends the support of a Borel measure to the support of the push-forward measure.
\end{proof}


\section{Chains and cone points}
\label{sec:chains}

Let $L= L_\varphi$ be the Liouville current associated to $\varphi \in \Flat(S)$.
An {\em $L$--chain} is a sequence of points (finite, half-infinite, or bi-infinite) 
\[ {\bf x} = (\ldots, x_0,x_1,\ldots ) \subset S^1_\infty \]
such that for all $i$ we have
\begin{enumerate}
\item $\{x_i,x_{i+1} \} \in \supp(L)$, and
\item $x_i,x_{i+1},x_{i+2}$ is counterclockwise ordered triple of distinct points and 
\[ \big[\{x_i,x_{i+1}\},\{x_{i+1},x_{i+2} \}\big] \cap \supp(L) = \{ \{x_i,x_{i+1}\},\{x_{i+1},x_{i+2} \} \}.\]
\end{enumerate}
Recall that $\big[\{x_i,x_{i+1}\},\{x_{i+1},x_{i+2} \}\big]$ is the set $\{x,y\} \in \G(\widetilde S)$ between $\{x_i,x_{i+1}\}$ and $\{x_{i+1},x_{i+2}\}$ as in Section~\ref{sec:link-between}.  For the motivation for condition 2, see Lemma~\ref{L:cone point char}.

Two chains that differ by a shift of the indices are considered the same.



\subsection{Constructing chains}

Fix a cone point $\zeta \in \cone(\varphi)$ in $\widetilde S$.  Let $\G(\varphi,\zeta)$ denote those $\varphi$--geodesics in $\G^*(\varphi)$ containing $\zeta$.  Let $\G^1(\varphi,\zeta) = \G(\varphi,\zeta) \setminus \partial_\varphi^{-1}(\G^2_\varphi(\widetilde S))$.  Alternatively, $\G^1(\varphi,\zeta)$ consists of those $\delta \in \G(\varphi,\zeta)$ which contain no other cone points besides $\zeta$, and which are asymptotic only to geodesics containing at most one cone point.

\begin{figure}[htb]
\centering
\includegraphics[width=4cm]{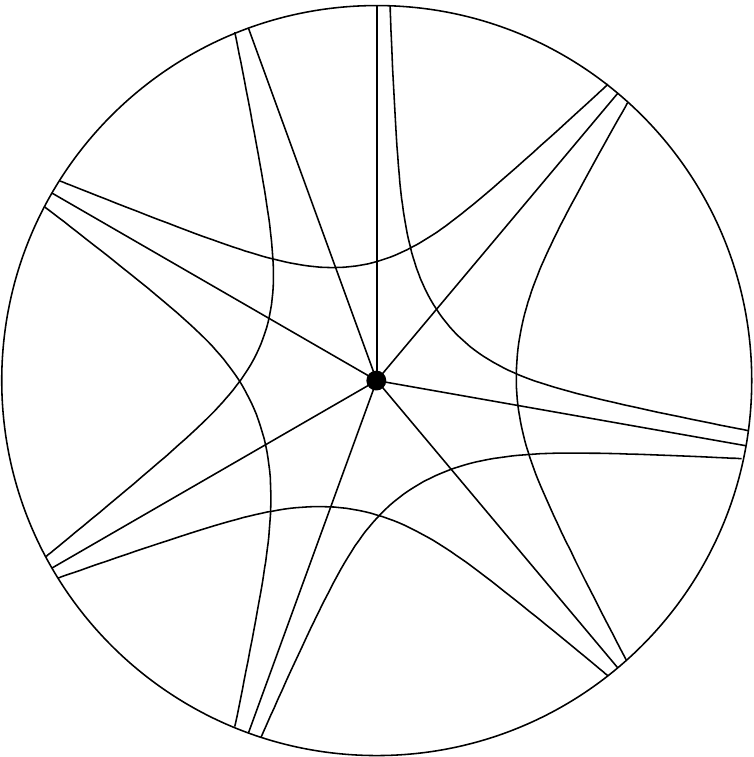}
\begin{picture}(0,0)
\put(-87,114){$x_0$}
\put(-121,25){$x_1$}
\put(-24,7){$x_2$}
\put(-23,103){$x_3$}
\put(-121,86){$x_4$}
\put(-87,-3){$x_5$}
\put(-2,45){$x_6$}
\put(-62,118){$x_7$}
\put(-75,55){$\zeta$}
\end{picture}
\caption{A finite chain $(x_0,x_1,\ldots,x_7) = \partial_\varphi(\delta_1,\ldots,\delta_7)$.  The geodesics $\delta_j$ are constructed from the rays emanating from the cone point $\zeta$.  We have also drawn nonsingular geodesics approximating each $\delta_j$ (from the side of $\delta_j$ where the angle is $\pi$) these to clarify which pairs of rays are used for each $\delta_j$.} \label{F:chain1}
\end{figure}
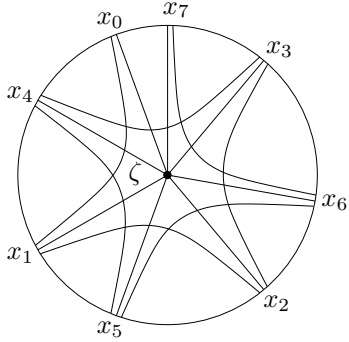

Declare two geodesics $\delta_1,\delta_2 \in \G^1(\varphi,\zeta)$ to be {\em adjacent at $\zeta$} if they are asymptotic in one direction.  We make a few simple observations.  First, adjacent geodesics must be $\varphi$--cone point asymptotic: since they contain the point $\zeta$ they cannot be separated by a $\varphi$--flat half-strip; see Lemma \ref{L:ray-asympt-class}.  Second, any $\delta \in \G^1(\varphi,\zeta)$ can be adjacent at $\zeta$ to at most two other geodesics in $\G^1(\varphi,\zeta)$: $\delta$ is made up of exactly two rays emanating from $\zeta$ and each of these is the common ray of at most one other geodesic in $\G^1(\varphi,\zeta)$.  

If $\delta_1,\delta_2 \in \G^1(\varphi,\zeta)$ are adjacent at $\zeta$ with endpoints $\partial_\varphi(\delta_i) = \{x,x_i\}$ for $i=1,2$, and $x_1,x,x_2$ are a counterclockwise ordered triple around $S^1_\infty$, then we write $\delta_1 \leq \delta_2$.  Note that by Lemma~\ref{L:cone point char}, the triple $x_1,x,x_2$ satisfy condition 2 in the definition of chain.

Now suppose $\boldsymbol \delta = ( \ldots, \delta_i,\delta_{i+1},\ldots ) \subset \G^1(\varphi,\zeta)$ is a sequence (finite, half-infinite, or bi-infinite) such that for all $i$ we have $\delta_i \leq \delta_{i+1}$.  Then let ${\bf x} = (\ldots,x_i,x_{i+1},\ldots)$ be the associated sequence of common endpoints;  see Figure~\ref{F:chain1}.  When $\boldsymbol \delta$ is bi-infinite (the primary case of interest), this is given by
\[ x_i = \partial_\varphi(\delta_i) \cap \partial_\varphi(\delta_{i+1}). \]
The same formula is valid in the finite or half-infinite case, except for the first and/or last terms of ${\bf x}$.  We write $\partial_\varphi(\boldsymbol \delta) = {\bf x}$, which is a chain.

\begin{proposition} \label{P:chains from cone points}
Suppose ${\bf x}$ is an $L$--chain with at least $3$ terms such that all consecutive pairs $\{x_i,x_{i+1}\}$ are in $\G^*_\varphi(\widetilde S) \setminus \G^2_\varphi(\widetilde S)$.  Then there exists a unique cone point $\zeta$ and sequence $\boldsymbol \delta \subset \G^1(\varphi,\zeta)$ such that $\partial_\varphi(\boldsymbol \delta) = {\bf x}$.
\end{proposition}
\begin{proof}  Suppose that $\{x_i,x_{i+1},x_{i+2}\}$ are three consecutive terms.  From the second condition in the definition of chain, Lemma~\ref{L:cone point char} guarantees a unique pair $\delta_i,\delta_{i+1}' \in \G^*(\varphi) \setminus \G^2(\varphi)$ such that $\partial_\varphi(\delta_i) = \{x_i,x_{i+1}\}$ and $\partial_\varphi(\delta_{i+1}') = \{x_{i+1},x_{i+2}\}$, and such that $\delta_i,\delta_{i+1}'$ are $\varphi$--cone point asymptotic.  Let $\zeta_i$ be the unique cone point contained in $\delta_i$ and $\delta_{i+1}'$.

\begin{claim} For each $i$, $\delta_{i+1}' = \delta_{i+1}$ and $\zeta_i = \zeta_{i+1}$.
\end{claim}
\begin{proof}[Proof of claim.] We note that $\delta_{i+1}'$ and $\delta_{i+1}$ are asymptotic (since they have the same boundary points).  Consequently, they are either equal and $\zeta_i = \zeta_{i+1}$ (being the unique cone point in the geodesic), or they are boundary geodesics of a $\varphi$--flat strip in $\widetilde S$.  In the former case, we are done, and hence we assume the latter case.  This implies that the ray $r = \delta_i \setminus \delta_{i+1}'$ of $\delta_i$ must lie on the {\em opposite} side of $\delta_{i+1}'$ as the side where the cone angle at $\zeta_i$ is $\pi$.  Similarly, $r' = \delta_{i+1} \setminus \delta_{i+2}'$ lies on the opposite side of $\delta_{i+1}$ as the side where the cone angle at $\zeta_{i+1}$ is $\pi$.  Note that the side of $\delta_{i+1}$ (respectively, $\delta_{i+1}'$) containing the $\varphi$--flat strip is the side where the angle at the cone points is $\pi$.  It follows that $x_i$ and $x_{i+3}$ must lie in different components of $S^1_\infty \setminus \{x_{i+1},x_{i+2} \}$.  But this contradicts the counterclockwise order around $S^1_\infty$ for consecutive triples.  See Figure \ref{F:chain2}.\end{proof}

\begin{figure}[htb]
\centering
\includegraphics[width=4cm]{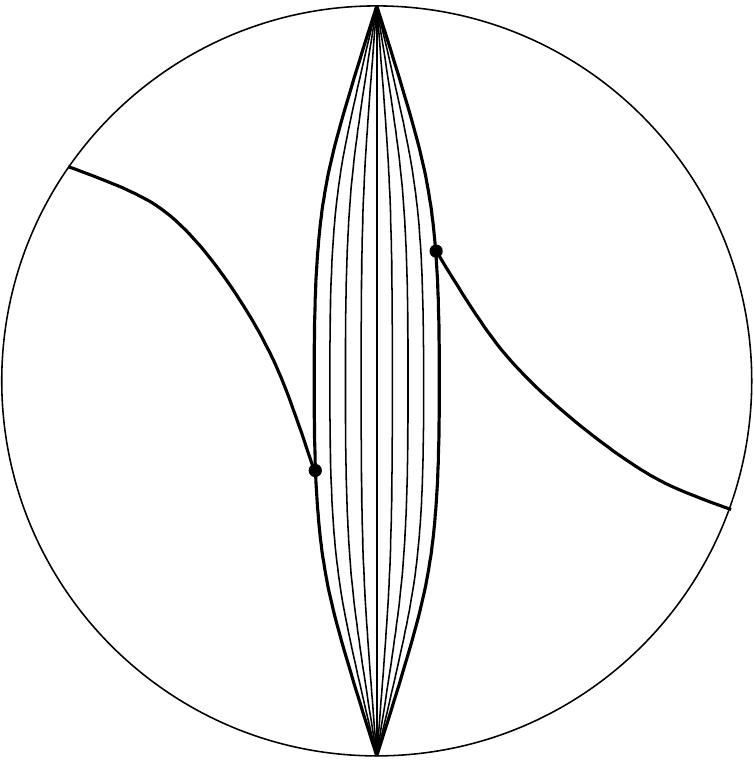}
\begin{picture}(0,0)
\put(-79,40){$\zeta_i$}
\put(-47,74){$\zeta_{i+1}$}
\put(-117,91){$x_i$}
\put(-67,-7){$x_{i+1}$}
\put(-67,118){$x_{i+2}$}
\put(-5,33){$x_{i+3}$}
\put(-91,70){$r$}
\put(-25,50){$r'$}
\end{picture}
\caption{The $\varphi$--flat strip has boundary geodesics $\delta_{i+1},\delta_{i+1}'$.  The strip forces one of the triples $x_i,x_{i+1},x_{i+2}$ or $x_{i+1},x_{i+2},x_{i+3}$ to be {\em clockwise} ordered.} \label{F:chain2}
\end{figure}
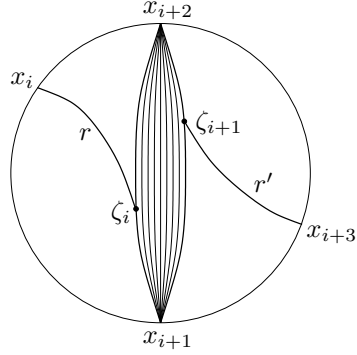

By the claim, there is a single cone point $\zeta$ so that $\zeta_i = \zeta$ for all $i$, and that
$ \partial_\varphi(\ldots,\delta_i,\delta_{i+1},\ldots ) = {\bf x}$.  Since each $\{x_i,x_{i+1}\}$ is in $\G^*_\varphi(\widetilde S) \setminus \G^2_\varphi(\widetilde S)$, it follows that $\delta_i \in \G^1(\varphi,\zeta)$.  Uniqueness of $\zeta$ follows from uniqueness in Lemma~\ref{L:cone point char}
\end{proof}

\subsection{Chain types and cone point identification}

Continue to assume $L = L_\varphi$ for some $\varphi \in \Flat(S)$.
For any countable set $\Omega \subset \G^*_\varphi(\widetilde S)$ with $\G^2_\varphi(\widetilde S) \subset \Omega$ (c.f.~Corollary~\ref{C:countable}), we define an {\em $(L,\Omega)$--chain} to be a bi-infinite $L$--chain ${\bf x}$ such that the consecutive pair $\{x_i,x_{i+1}\}$ is in $\G^*_\varphi(\widetilde S) \setminus \Omega$ for all $i \in \Z$.  Let $\Ch(L,\Omega)$ denote the set of all $(L,\Omega)$--chains.

Since $\G^2_\varphi(\widetilde S) \subset \Omega$ Proposition~\ref{P:chains from cone points} implies that every $(L,\Omega)$--chain ${\bf x}$ is given by ${\bf x} = \partial_\varphi(\boldsymbol \delta)$ for some unique $\zeta$ and $\boldsymbol \delta \subset \G^1(\varphi,\zeta)$.   We therefore have a well-defined map
\[ \partial_{\varphi}^\# \colon \Ch(L,\Omega) \to \cone(\varphi).\]
\begin{lemma}  \label{L:nonempty chains} 
For any countable $\Omega \supset \mathcal G^2_\varphi(\widetilde S)$, $\partial_\varphi^\#$ is surjective.
\end{lemma}
\begin{proof}  Fix any $\zeta \in \cone(\varphi)$ in $\widetilde S$.  Consider the set of $\varphi$--geodesics $\delta \in \G(\varphi)$ so that for one of the sides of $\delta$, it makes an angle $\pi$ on that side at every cone point it contains.  Such geodesics are parameterized by pairs of directions at $\zeta$ making angle $\pi$.  There are uncountably many such, with only countably many either containing more than one cone point or being asymptotic to a geodesic with more than one cone point.  Let $\Delta$ be the remaining uncountably set of geodesics.  All but countably many geodesics in $\Delta$ are contained in a unique bi-infinite sequence $\boldsymbol \delta = (\ldots,\delta_i,\delta_{i+1},\ldots)$  with each $\delta_i \in \Delta$ (note that any such bi-infinite sequence is uniquely determined by any of its elements).  There are uncountably many such sequences.  At most countably many of these can contain a geodesic $\delta$ with $\partial_\varphi(\delta) \in \Omega$.  The remaining uncountably many sequences $\boldsymbol \delta$ have $\partial_\varphi(\delta) \in \Ch(L,\Omega)$.  Applying $\partial_\varphi^\#$ to any of these we get $\zeta$.  Since $\zeta$ was arbitrary, $\partial_\varphi^\#$ is onto.
\end{proof}

As this proof illustrates, a sequence $\boldsymbol \delta = (\ldots,\delta_i,\delta_{i+1},\ldots)$ with $\partial_\varphi(\boldsymbol \delta) \in \Ch(L,\Omega)$ with $\partial_\varphi^\#({\bf x}) = \zeta$ is determined by a sequence of rays emanating from $\zeta$, so that consecutive rays make counterclockwise angle $\pi$.  The next lemma is clear from this.
\begin{lemma} \label{L:periodic aperiodic}
Any $(L,\Omega)$--chain ${\bf x}$ is either periodic or the sequence consists of distinct points in $S^1_\infty$.  These two cases correspond to the case when the $\varphi$--cone angle at $\partial^\#_\varphi({\bf x})$ is a rational multiple of $\pi$ and irrational multiple of $\pi$, respectively. \qed
\end{lemma}
The second case, when all points in ${\bf x}$ are distinct, we will call {\em aperiodic}.

The next lemma tells us that we can decide when two chains have the same $\partial_\varphi^\#$--image, appealing only to topological properties of $\supp(L)$.  To describe it, we first make a definition for periodic chains.

Any periodic $(L,\Omega$)--chain ${\bf x}$ contains exactly $k = k({\bf x}) \geq 3$ points in $S^1_\infty$ (repeated infinitely often)---this is precisely the minimal period of the sequence.  There is thus a smallest integer $n = n({\bf x}) > 0$ so that the sequence
\[ x_0,x_n,x_{2n},\ldots,x_{(k-1)n} \]
is this set of points cyclically ordered counterclockwise.  To see this, note that $x_i,x_{i+1}$ are endpoints of rays emanating from $\zeta$ making a counterclockwise angle $\pi$.  There are finitely many such rays by periodicity, and any two rays which are counterclockwise consecutive, make the same angle.  The shift $x_i \mapsto x_{i+1}$ generates a group of rotations of these rays through an angle $\pi$, and this acts transitively on the rays.  Thus, rotating each ray to the next counterclockwise consecutive ray is some minimal power $n$ of our generator, and thus $x_i$ and $x_{i+n}$ are counterclockwise consecutive points in ${\bf x}$ around $S^1_\infty$.
See Figure~\ref{F:chain3}.

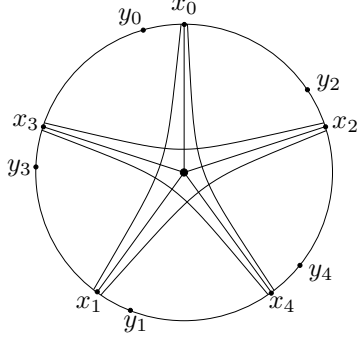
\begin{figure}[htb]
\centering
\includegraphics[width=4cm]{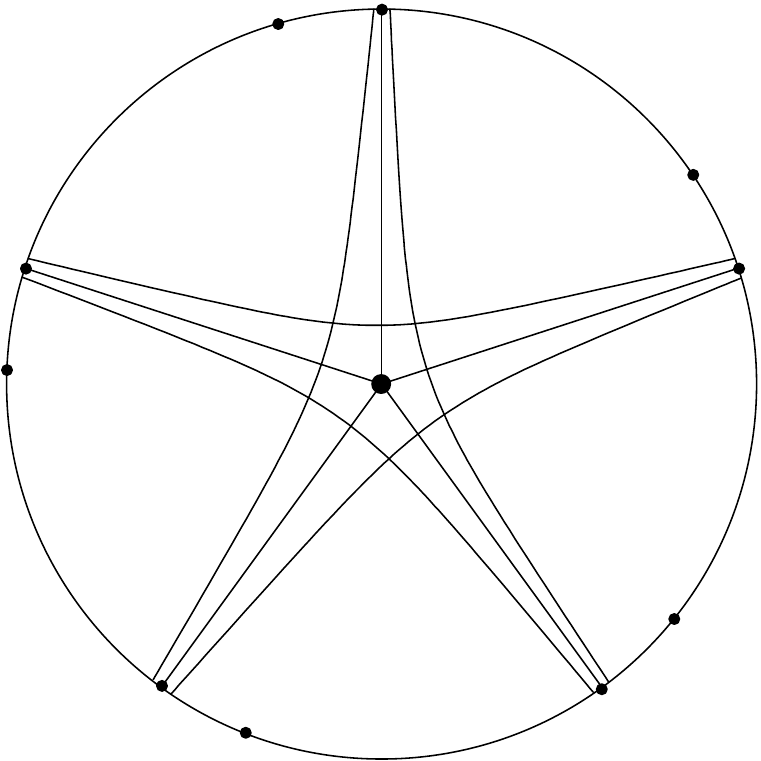}
\begin{picture}(0,0)
\put(-65,118){$x_0$}
\put(-101,5){$x_1$}
\put(-4,73){$x_2$}
\put(-125,74){$x_3$}
\put(-28,5){$x_4$}
\put(-85,114){$y_0$}
\put(-83,-2){$y_1$}
\put(-10,89){$y_2$}
\put(-127,57){$y_3$}
\put(-13,17){$y_4$}
\end{picture}
\caption{Endpoints of a periodic chain ${\bf x} = (\ldots,x_0,x_1,x_2,x_3,x_4,\ldots)$ with period $5$. The periodic chain ${\bf y} = (\ldots,y_0,y_1,y_2,y_3,y_4,\ldots)$ shown, also with period $5$, is perfectly interlaced with ${\bf x}$.} \label{F:chain3}
\end{figure}

We say that periodic chains ${\bf x}$ and ${\bf y}$ are {\em perfectly interlaced} if $k({\bf x}) = k({\bf y}) = k$ and $n({\bf x}) = n({\bf y})=n$ and after shifting indices of ${\bf y}$ if necessary we have $y_{jn}$ lies in the counterclockwise interval $(x_{jn},x_{(j+1)n})$ for all $j \in \Z$.  See Figure~\ref{F:chain3}.  By periodicity, this implies $y_{jn+r}$ lies in the counterclockwise interval $(x_{jn+r},x_{(j+1)n+r})$ for all $j,r \in \Z$.

\begin{lemma} \label{L:Intrinsic cone points}
Given ${\bf x},{\bf y} \in \Ch(L,\Omega)$ we have $\partial_\varphi^\#({\bf x}) = \partial_\varphi^\#({\bf y})$ if and only if
\begin{enumerate}
\item[(a)] ${\bf x}$ and ${\bf y}$ are both periodic and their endpoints are perfectly interlaced, or
\item[(b)] ${\bf x}$ and ${\bf y}$ are both aperiodic and for any $y_i,y_{i+1}$, there exists a sequence $x_{j_n},x_{j_n+1} \to y_i,y_{i+1}$ as $j_n \to \infty$.
\end{enumerate}
\end{lemma}
\begin{proof}
Assume first that $\zeta = \partial_\varphi^\#({\bf x}) = \partial_\varphi^\#({\bf y})$.  Both are either periodic or both are aperiodic as this only depends on the cone angle at the point $\zeta$.  If they are both periodic, then the rays emanating from $\zeta$ defining ${\bf x}$ differ from those defining ${\bf y}$ by rotating by some fixed angle.  This clearly implies that ${\bf x}$ and ${\bf y}$ are perfectly interlaced.  If both are aperiodic, and $\partial_\varphi(\boldsymbol \delta) = {\bf x}$, then note that the set of geodesics in $\delta_i \in \boldsymbol \delta$ are dense in the set of geodesics $\G^1(\varphi,\zeta)$.  In particular for any consecutive pair $y_i,y_{i+1}$, let $\delta \in \G^1(\varphi,\zeta)$ be such that $\partial_\varphi(\delta)$.   Then there is a sequence $\delta_{j_n}$ with $\delta_{j_n} \to \delta$.  This implies condition (b).
We now prove the reverse implication.

For case (a), we suppose ${\bf x},{\bf y} \in \Ch(L,\Omega)$ are perfectly interlaced and $\zeta_{\bf x} = \partial_\varphi^\#({\bf x})  \neq  \partial_\varphi^\#({\bf y}) =\zeta_{\bf y}$.  Let $\boldsymbol \delta_{\bf x}$ and $\boldsymbol \delta_{\bf y}$ with $\partial_\varphi(\boldsymbol \delta_{\bf x}) = {\bf x}$ and $\partial_\varphi(\boldsymbol \delta_{\bf y}) = {\bf y}$.  For each $\delta_{i,{\bf y}} \in \boldsymbol \delta_{\bf y}$, let $H_{i,{\bf y}}^+$ denote the half-plane in $\widetilde S$ bounded by $\delta_{i,{\bf y}}$ containing the side on which $\delta_{i,{\bf y}}$ makes angle $\pi$ at $\zeta_{\bf y}$.  Observe that $\widetilde S= \cup_i H_{i,{\bf y}}^+$, and thus $\zeta_{\bf x} \in H_{i,{\bf y}}^+$, for some $i$.  Without loss of generality, suppose $i=0$, and $y_{jn+r} \in (x_{jn+r},x_{(j+1)n+r})$ for all $j,r \in \Z$.   In particular $y_0 \in (x_0,x_n)$ and $y_1 \in (x_1,x_{n+1})$.  The $\varphi$--geodesic between $x_0$ and $x_{n+1}$ passes through $\zeta_{\bf x}$ (though it may make angle greater than $\pi$ on both sides, and hence may not be in $\boldsymbol \delta_{\bf x}$).  On the other hand, the endpoints of this geodesic are in the closure of the {\em complementary half-plane} $\widetilde S \setminus H_{0,{\bf y}}^+$.  Thus this geodesic must meet $\delta_{0,{\bf y}}$ in two points, contradicting the fact that $\varphi$ is $\CAT(0)$.   See Figure~\ref{F:chain4}.

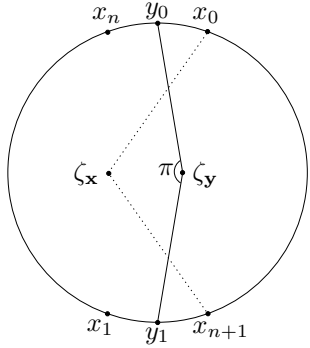
\begin{figure}[htb]
\centering
\includegraphics[width=4cm]{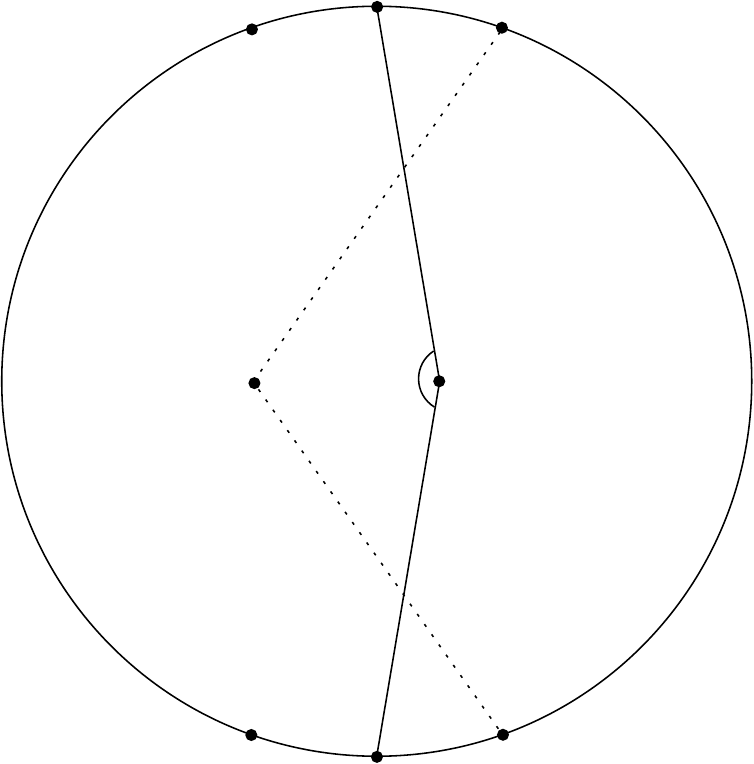}
\begin{picture}(0,0)
\put(-65,118){$y_0$}
\put(-65,-5){$y_1$}
\put(-47,116){$x_0$}
\put(-47,-3){$x_{n+1}$}
\put(-47,55){$\zeta_{\bf y}$}
\put(-87,-2){$x_1$}
\put(-85,116){$x_n$}
\put(-92,55){$\zeta_{\bf x}$}
\put(-60,57){$\pi$}
\end{picture}
\caption{Concatenating the ray from $\zeta_{\bf x}$ to $x_0$ with the ray from $\zeta_{\bf x}$ to $x_{n+1}$ is a geodesic.} \label{F:chain4}
\end{figure}

For case (b), we use the same notation as in case (a) writing $\boldsymbol \delta_{\bf x}$ and $\boldsymbol \delta_{\bf y}$ with $\partial_\varphi(\boldsymbol \delta_{\bf x}) = {\bf x}$ and $\partial_\varphi(\boldsymbol \delta_{\bf y}) = {\bf y}$, and $H_{i,{\bf y}}^+$ for the half-planes bounded by $\delta_{i,{\bf y}}$ containing the sides with angle $\pi$ at $\zeta_{\bf y}$.  In this case, we choose $i$ so that $\zeta_{\bf y}$ lies in the complementary half-plane $\zeta_{\bf y} \in \widetilde S \setminus H_{i,{\bf y}}^+$.  Condition (b) guarantees a sequence $\{\{x_{j_n},x_{j_n+1}\}\}_n = \{\partial_\varphi(\delta_{j_n,{\bf x}})\}_n$ so that $\{x_{j_n},x_{j_n+1}\}$ converges to $\{y_i,y_{i+1}\}$.  It follows that $\delta_{j_n,{\bf x}}$ converges to a geodesic $\delta$ which is asymptotic to $\delta_{i,{\bf y}}$.  However, $\zeta_{\bf x} \in \delta$ (since it lies in all approximating geodesics) while $\zeta_{\bf y}$ is the only cone point in $\delta_{i,{\bf y}}$.  Consequently there is a $\varphi$--flat strip between $\delta$ and $\delta_{i,{\bf y}}$.  But this must be in the closure of $\zeta_{\bf y} \in \widetilde S \setminus H_{i,{\bf y}}^+$, which is impossible since the cone angle at $\zeta_{\bf y}$ on that side of $\delta_{i,{\bf y}}$ is strictly greater than $\pi$.
\end{proof}

Given two $(L,\Omega)$--chains ${\bf x}$ and ${\bf y}$, we write ${\bf x} \sim {\bf y}$ if and only if $\partial_\varphi^\#({\bf x}) = \partial_\varphi^\#({\bf y})$.  According to Lemma~\ref{L:Intrinsic cone points}, we see that ${\bf x} \sim {\bf y}$ is determined by $\supp(L)$, without reference of $\varphi$.  We write $[{\bf x}]$ for the equivalence class of ${\bf x}$.

\subsection{Chains and distances}

Continue to assume $L= L_\varphi$ for $\varphi \in \Flat(S)$ and $\Omega \subset \supp(L)$ a countable set containing $\G^2_\varphi(\widetilde S)$.  As we have already seen, the data of $(L,\Omega)$ can be used to reconstruct the cone points for the metric $\varphi$. In this section, we continue to read off information about $\varphi$ from $L$.  

For two distinct equivalence classes of $(L,\Omega)$--chains $[{\bf x}]$ and $[{\bf y}]$, we define
\[ \big[ [{\bf x}],[{\bf y}] \big] = \bigcup [\{x_i,x_{i+1}\},\{y_j,y_{j+1}\} ] \]
where the union is over ${\bf x}' \in [{\bf x}]$, ${\bf y}' \in [{\bf y}]$, and all consecutive pairs $\{x_i,x_{i+1}\}$ in ${\bf x}'$ and $\{y_j,y_{j+1}\}$ in ${\bf y}'$.  In words, $\big[ [{\bf x}],[{\bf y}] \big]$ consists of all endpoints of geodesics that lie between some geodesic from some ${\bf x}'$ in $[{\bf x}]$ and some geodesic from some ${\bf y}'$ in $[{\bf y}]$.
\begin{proposition} \label{P:measure is distance}
For any pair of distinct equivalence classes $[{\bf x}],[{\bf y}]$ with $\zeta_{\bf x} = \partial_\varphi^\#({\bf x})$ and $\zeta_{\bf y} = \partial_\varphi^\#({\bf y})$ we have
\[ L\left(\big[ [{\bf x}],[{\bf y}] \big] \right) = \varphi(\zeta_{\bf x},\zeta_{\bf y}).\]
\end{proposition}
\begin{proof}
We first observe that $\big[ [{\bf x}],[{\bf y}] \big] \cap \G^\circ_\varphi(\widetilde S)$ is precisely the $\partial_\varphi$--image of the set $E^\circ(\alpha)$ of all $\varphi$--geodesics in $\G^0(\varphi)$ transversely crossing the $\varphi$--geodesic segment $\alpha$ between $\zeta_{\bf x}$ and $\zeta_{\bf y}$.
Since $\G^\circ(\varphi)$ is a set of full $\hat L_\varphi$--measure, so $\G^\circ_\varphi(\widetilde S)$ is a set of full $L$--measure.  Therefore, appealing to Proposition~\ref{P:pre-measure is length} we have
\[ L\left(\big[ [{\bf x}],[{\bf y}] \big] \right) = L\left(\big[ [{\bf x}],[{\bf y}] \big] \cap \G^\circ_\varphi(\widetilde S) \right) = \hat L_\varphi(E^\circ (\alpha))=\ell_\varphi(\alpha) = \varphi(\zeta_{\bf x},\zeta_{\bf y}).\]
\end{proof}

\section{Proof of the Main Theorem}
\label{sec:end}

The Main Theorem will follow easily from the next

\begin{theorem} \label{T:main2} Suppose $\varphi_1,\varphi_2 \in \Flat(S)$ with $L_{\varphi_1} = L = L_{\varphi_2}$.  Then $\varphi_1 = \varphi_2$ in $\Flat(S)$, i.e.~there is an isometry $(S,\varphi_1) \to (S,\varphi_2)$ isotopic to the identity.
\end{theorem}
\begin{proof}
Let $\Omega = \G^2_{\varphi_1}(\widetilde S) \cup \G^2_{\varphi_2}(\widetilde S)$.  According to Lemma~\ref{L:nonempty chains}, for each $i=1,2$ the map $\partial_{\varphi_i}^\# \colon \Ch(L,\Omega) \to \cone(\varphi_i)$ is surjective.  Appealing to Lemma~\ref{L:Intrinsic cone points} we have 
\[ \partial_{\varphi_1}^\#({\bf x}) = \partial_{\varphi_1}^\#({\bf y}) \quad \Leftrightarrow \quad {\bf x} \sim {\bf y} \quad \Leftrightarrow \quad \partial_{\varphi_2}^\#({\bf x}) = \partial_{\varphi_2}^\#({\bf y})\]
Consequently, sending $\partial_{\varphi_1}^\#({\bf x})$ to $\partial_{\varphi_2}^\#({\bf x})$ well-defines a bijection
\[ F \colon \cone(\varphi_1) \to \cone(\varphi_2).\]
independent of the choice of ${\bf x}$ within the equivalence class.  Both $\partial_{\varphi_1}^\#$ and $\partial_{\varphi_2}^\#$ are $\pi_1(S)$--equivariant, and so $F$ is also.  Furthermore, according to Proposition~\ref{P:measure is distance} we have
\begin{equation} \label{E:isometry on vertices}
\varphi_2(F(\zeta),F(\zeta')) = \varphi_1(\zeta,\zeta').
\end{equation}

Now without loss of generality, assume that $Area_{\varphi_1}(S)\leq Area_{\varphi_2}(S)$.
Let $\mathcal{T}$ be an $\varphi_1$--triangulation of $S$ such that $\cone(\varphi_1)$ are the vertices of $\mathcal{T}$. Lift $\mathcal{T}$ to a triangulation $\widetilde{\mathcal{T}}$ on $\widetilde S$ and define a map
\[ \widetilde f \colon (\widetilde S,\varphi_1) \to (\widetilde S,\varphi_2) \]
as follows.  First, define $\widetilde f$ on $\cone(\varphi_1)$ by $\widetilde f|_{\cone(\varphi_1)} = F$.  This is $\pi_1(S)$--equivariant, and so we extend over edges of triangles of $\widetilde{\mathcal{T}}$, $\pi_1(S)$--equivariantly mapping these to geodesics.  Note that by Equation~(\ref{E:isometry on vertices}), we can assume that the restriction to every edge is also an isometry.

Let $\widetilde \Delta$ denote a lift of a triangle in $\mathcal{T}$ and $\widetilde \Delta'=\widetilde f(\widetilde \Delta)$. The universal cover $(\widetilde S,\varphi_1)$ is a $\CAT(0)$ space and since all cone points are vertices, we conclude that $\widetilde \Delta$ contains no cone points besides its vertices, and hence is itself a comparison triangle for $\widetilde \Delta'$. Thus by Proposition~\ref{prop:CAT(0)}, $Area(\widetilde \Delta') \leq Area(\widetilde \Delta)$ for every triangle $\widetilde \Delta$ and its image $\widetilde \Delta' = \widetilde f(\widetilde \Delta)$.    Consequently, applying this to every triangle in $\mathcal{T}$ we get $Area_{\varphi_1}(S) \geq Area_{\varphi_2}(S)$. Therefore $Area_{\varphi_1}(S)= Area_{\varphi_2}(S)$ and $Area(\widetilde \Delta') = Area(\widetilde \Delta)$ for every triangle $\widetilde \Delta$. Appealing to Proposition~\ref{prop:CAT(0)}, it follows that we may extend $\widetilde f$ over every triangle by an isometry, and hence $\widetilde f$ is a $\pi_1(S)$--equivariant isometry.  This descends to an isometry $f \colon (S,\varphi_1) \to (S,\varphi_2)$.

All that remains is to prove that $f$ is isotopic to the identity.  However, the action of $\pi_1(S)$ on $S$ is independent of the metric $\varphi_i$.  In particular, $\pi_1(S)$--equivariance implies that the continuous extension of $\widetilde f$ to $S^1_\infty$ is the identity.  It follows that $f$ induces the identity on $\pi_1(S)$, and hence is isotopic to the identity.
\end{proof}

The main theorem from the introduction is now an easy corollary of this.
\begin{theoremmain}
If $\varphi_1,\varphi_2 \in \Flat(S)$ and $\Lambda(\varphi_1) = \Lambda(\varphi_2)$, then $\varphi_1 = \varphi_2$.
\end{theoremmain}

\begin{proof}
If $\Lambda(\varphi_1) = \Lambda(\varphi_2)$, then by Proposition~\ref{P:measure is length} and Theorem~\ref{theorem:Otal}, it follows that $L_{\varphi_1} = L_{\varphi_2}$.  Theorem~\ref{T:main2} completes the proof.
\end{proof}

\section{Open questions} \label{sec:questions}

If $\varphi \in \Flat(S)$ is any metric, we can scale by $a > 0$, and the Liouville current will scale the same: $L_{a\varphi} = aL_{\varphi}$.  In particular $\supp(L_{\varphi}) = \supp(L_{a\varphi})$.  If $\varphi$ is defined by a holomorphic quadratic differential, then for any $A \in SL_2(\mathbb R)$, one can deform the quadratic differential, and hence the metric, by $A$.  If we let $\varphi_A$ be such a deformation, then the identity $(S,\varphi) \to (S,\varphi_A)$ is {\em affine}:  there are locally isometric coordinates for $\varphi$ and $\varphi_A$ so that in these coordinates, the identity is affine (with derivative $A$).  The formula for the Liouville current in \cite{rafi:LD} shows that $L_\varphi$ and $L_{\varphi_A}$ will also have the same support.  The scaling deformation above is a special case which can be carried out for any $\varphi \in \Flat(S)$, but if $A$ is not simply a homothety, then this kind of deformation is special to those metrics defined by quadratic differentials.  This is because an eigenspace of the derivative of the affine map must be parallel on $(S,\varphi)$, and hence the holonomy must lie in $\{\pm I\}$.  We conjecture that this is the only way that two metrics in $\Flat(S)$ can have Liouville currents with the same support.

\begin{conjecture}  Given $\varphi_1,\varphi_2 \in \Flat(S)$, if $\supp(L_{\varphi_1}) = \supp(L_{\varphi_2})$ then there exists an affine map $f \colon (S,\varphi_1) \to (S,\varphi_2)$, isotopic to the identity.
\end{conjecture}

The proof of the Main Theorem shows that if the supports are the same, then there is a map isotopic to the identity sending cone points bijectively to cone points.  It is also straight forward to see this map actually preserves cone angles.  With a little more work, one can show that the metrics coming from quadratic differentials can be distinguished from all other metrics in $\Flat(S)$ in terms of the supports of their Liouville currents.  The work in \cite{rafi:LD} can then be used to prove the conjecture for quadratic differential metrics.  Since metrics coming from holomorphic $q$--differentials, $q \in \mathbb Z$, can also be distinguished from other flat metrics, this would be an interesting test case for the conjecture, but we have not carried out this analysis.

We also ask the following question which would provide a common generalization of the Main Theoem, as well as the marked length results in \cite{otal:length,croke,Croke:TM,hersonsky:OT,rafi:LD}.

\begin{question}  Are the nonpositively curved cone metrics spectrally rigid?
\end{question}

It seems likely that some combination of the techniques here and in \cite{Croke:TM} may be sufficient, but we have not investigated this.

  \bibliographystyle{alpha}
  \bibliography{main}

\bigskip

\noindent
{\tt Department of Mathematics, Boston College, Chestnut Hill, MA 02467, U.S.A.; https://www2.bc.edu/anja-bankovic/}\\  E-mail address$\colon$ {\tt anja.bankovic@bc.edu}

\bigskip

\noindent
{\tt  Department of Mathematics, University of Illinois at
  Urbana, IL 61801, U.S.A.;  http://www.math.uiuc.edu/\~{}clein/}\\ E-mail address$\colon$ {\tt clein@math.uiuc.edu}

\end{document}